\documentclass{llncs}
\begin{document}

\title{Injective Hulls In a Locally Finite Topos}
\titlerunning{Locally Finite Topoi}  
%
\author{Felix Dilke\inst{1}}
\authorrunning{Felix Dilke} 
%
\tocauthor{Felix Dilke}
\institute{SpringerNature,\\
\email{fdilke@gmail.com},\\ WWW home page:
\texttt{http://github.com/fdilke/bewl}
}

\maketitle              

\begin{abstract}
We show that in a locally finite topos, every object has an essential extension that is injective, and that this extension is unique up to isomorphism.

The construction was motivated by work on Bewl, a software project for doing topos-theoretic calculations.
\end{abstract}
\section{Motivation}
These results emerged from trying to construct injective hulls in software using Bewl (\url{http://github.com/fdilke/bewl}), a software domain-specific language for topos theory based on the Scala programming language. 

It seemed straightforward to use iterative algorithms to construct a minimal injective extension for any object. The question was how this can be theoretically justified. Of course, because topoi in Bewl are modelled on a finite computer, it is more or less enforced that they all satisfy the condition of local finiteness. 

So the results below show that this condition alone ensures that injective hulls exist in such a topos.
\section{Preliminaries}
\subsection{Locally finite categories}
Let $\mathcal{C}$ be a category which is $\textit{locally finite}$, i.e. in which for any two objects A, B there are only finitely many arrows A $\rightarrow$ B. 

\begin{lemma} \label{endo:monic}
Let A $\in \mathcal{C}$, $ f:A \rightarrow A $ monic. Then $f$ is an automorphism of A. 
\end{lemma}

\begin{proof}
Since all powers $f^n : A \rightarrow A$ lie in the finite set $Hom(A, A)$, we can find $m,  n \geq 0 \in  \bbbn $ with $f^{m+n+1} = f^m$. Since $f^m$ is monic, $f^{n + 1} = 1_A$ and so ${f^n}$ is a two-sided inverse of $f$. $\qed$
\end{proof}

Applying this result to the dual category $\mathcal{C}^{op}$, which is also locally finite, we note
\begin{corollary} \label{endo:epic}
Any epic endomorphism of a $\mathcal{C}$-object is iso. \qed
\end{corollary}

We also deduce a version of the Schr\"oder-Bernstein theorem for $\mathcal{C}$:

\begin{theorem} \label{schro:ben}
Let A, B $\in \mathcal{C}$ with monics $f:A \rightarrow B$, $g:B \rightarrow A$. Then f, g are iso. In particular, $A \cong B$. 
\end{theorem}

\begin{proof}
By the lemma, the monics $fg$ and $gf$ are both iso, so $f$ is both left- and right-invertible, hence iso. Similarly for $g$. $\qed$
\end{proof}

\begin{lemma}
Let $A_0 \rightarrow A_1 \rightarrow \cdots \rightarrow A_n$ be a sequence of monics in $\mathcal{C}$ whose endpoints are isomorphic, $A_0 \cong A_n$. Then each arrow in the sequence is iso. 
\end{lemma}

\begin{proof}
It is enough to prove this for $n=2$. 

So suppose $p: A_0 \rightarrow A_1, q: A_1 \rightarrow A_2$ are monic, $r: A_2 \rightarrow A_0$ an isomorphism. Then $qpr$ is an monic endomorphism of $A_2$, hence an isomorphism by Lemma $\ref{endo:monic}$. Let $s$ be its inverse. Then $qprs = 1_{A_2}$. But then $qprsq = q$, so cancelling $q$ from the left, $prsq = 1_{A_1}$. It follows that $q$, $prs$ are mutual two-sided inverses. Hence $q$ is iso, and since $r$, $s$ are both iso, it also follows that $p$ is iso. \qed
\end{proof}

Dually, we have

\begin{lemma}
Let $A_0 \rightarrow A_1 \rightarrow \cdots \rightarrow A_n$ be a sequence of epimorphisms in $\mathcal{C}$ whose endpoints are isomorphic, $A_0 \cong A_n$. Then each arrow in the sequence is iso. \qed
\end{lemma}

\subsection{Locally finite topoi}

Now let $\mathcal{E}$ be a locally finite topos. We shall need some basic results about sequences of monics and epics in $\mathcal{E}$.

\begin{lemma} \label{finite:images}
Each $A \in \mathcal{E}$ has only finitely many isomorphism classes of epimorphic images. More precisely, there is a finite set $\{ I_i \in \mathcal{E} \}_{i = 0}^{n}$ such that for any epi $f: A \rightarrow I$, we have some $I_i \cong I$.
\end{lemma}

\begin{proof}
We recall that in a topos every epi is the coequalizer of its kernel pair (theorem IV.7.8 in \cite{moer:mac}), the kernel pair being a pair of jointly monic arrows $p, q: K \rightarrow A$ for some $K$, or equivalently a subobject $K \rightarrow A \times A$. But each such subobject is determined up to isomorphism by its characteristic arrow $A \times A \rightarrow \Omega$, and there are only finitely many of these.

So the codomain $B$ of any epi $A \rightarrow B$ can be recovered up to isomorphism as the coequalizer of one of these finitely many possible kernel pairs. \qed
\end{proof}

\begin{theorem} \label{seq:epic}
Let $A_{n \in \bbbn} \in \mathcal{E}$ be a sequence of objects, $f_{n \in \bbbn}$ a sequence of epis with each $f_n: A_n \rightarrow A_{n+1}$. Then all $f_n$ are iso from some point on.
\end{theorem}

\begin{proof}
By Lemma $\ref{finite:images}$, we can find a finite set $\{ I_i \}_{i = 0}^{n}$ of representatives for isomorphism classes of epimorphic images of $A_0$. But then every $A_j$ is isomorphic to some $I_i$. By the pigeon-hole principle, we can find a sequence ${n_0 < n_1 < ... \in \bbbn}$ such that all $A_{n_k}$ are isomorphic. But then each subsequence of epis $A_{n_k} \rightarrow \cdots \rightarrow A_{n_{k+1}}$ satisfies the hypotheses of Lemma 3, and so each of its component arrows is iso. 
This shows that all $f_n$ are isomorphisms for $n \geq n_0$. \qed
\end{proof}

\begin{theorem}
Let $E \in \mathcal{E}$, $A_0 \rightarrow A_1 \rightarrow \cdots$ a sequence of monics where each $A_n$ is a subobject of $E$. Then each arrow in the sequence is iso from some point on.
\end{theorem}

\begin{proof}
Since $E$ has only finitely many isomorphism classes of subobjects, we can again find a sequence ${n_0 < n_1 < ... \in \bbbn}$ such that all $A_{n_k}$ are isomorphic. We can then argue as in the proof of Theorem $\ref{seq:epic}$. \qed
\end{proof}

\subsection {Injective objects in a topos}
We recall that in any topos, the subobject classifier $\Omega$ is injective, and the exponential object $B^A$ is injective whenever $B$ is injective. Hence for any $A$, $\Omega^A$ is injective, and the singleton map $\{\}:A \rightarrow \Omega^A$ therefore provides an embedding of $A$ into an injective object.

Note also that any retract of an injective object is injective.

\subsection {Essential extensions in a topos}
We recall that a monic $e: A \rightarrow B$ is $\textit{essential}$, or an $\textit{essential extension}$, if for every $C \in \mathcal{E}$ and $g: B \rightarrow C$, $ge$ monic implies $g$ monic. 
Clearly a composite of essential extensions is essential. 

We shall also need an alternative description in terms of epimorphisms, which relies on the fact that any arrow in a topos can be expressed as a product of a monic and an epic (\cite{mclarty} theorem 16.4).
\begin{lemma} \label{ess:epi}
In a topos, a monic $e: A \rightarrow B$ is essential iff for any epic $g: B \rightarrow C$, $ge$ monic $\Rightarrow$ $g$ iso.
\end{lemma}

\begin{proof}
$\Rightarrow$: Clearly $g$ is monic; being epic as well, it must be iso (\cite{mclarty} theorem 13.1).
$\Leftarrow$: Take an arbitrary $g: B \rightarrow C$ with $ge$ monic. Factorize $g = pq$ with $p: X \rightarrow C$, $q: B \rightarrow X$ epic. Then $ge = pqe$ is monic. Since $p$ is monic, $qe$ is monic, forcing $q$ to be monic and hence iso. This shows that $g = pq$ is monic. \qed
\end{proof}

\section{Injective hulls in $\mathcal{E}$}

\begin{theorem} \label{inj:noesse}
An object $A \in \mathcal{E}$ is injective iff every essential extension of A is iso. 
\end{theorem}

\begin{proof}
$\Rightarrow$: Let $e: A \rightarrow B$ be essential. Then by injectivity, the identity map $1_A: A \rightarrow A$ can be extended along $e$ to an arrow $f: B \rightarrow A$, i.e. $fe = 1_A$. But then $fe$ is monic, so invoking the fact that $e$ is essential, $f$ is monic. We now have monics $e$, $f$ in each direction between $A$ and $B$. Invoking Theorem $\ref{seq:epic}$, in particular, $e$ is iso.

$\Leftarrow$: 
We show that any monic $e: A \rightarrow E_0$ can be retracted onto A. This will prove the result because we can take $E_0$ to be injective.

The proof is by contradiction. Suppose $e$ has no left inverse. Then in particular it is not iso, hence not an essential extension. 

Applying Lemma \ref{ess:epi}, we can therefore find a proper epic (i.e. non-iso) $q: E_0 \rightarrow E_1$ with $qe$ monic. 

If $qe$ were iso, we could now construct a left inverse to $e$, contrary to hypothesis. The argument of the last paragraph can therefore be repeated with $qe$ in place of $e$, yielding a sequence of proper epics $E_0 \rightarrow E_1 \rightarrow E_2$ whose composite with $e$ is monic. Continuing in this way, we find an infinite sequence of proper epics $E_0 \rightarrow E_1 \rightarrow \cdots$, contradicting Theorem $\ref{seq:epic}$. \qed
\end{proof}

\begin{theorem} [Existence and uniqueness of injective hulls]
For any $A \in \mathcal{E}$,
\begin{enumerate}
  \item There is an essential extension $e: A \rightarrow E$ with $E$ injective.
  \item For any other essential extension $f: A \rightarrow F$ with $F$ injective, there is an isomorphism $h: E \cong F$ with $f = he$.
\end{enumerate}
\end{theorem}

\begin{proof}
\begin{enumerate}
  \item Find a monic $A \rightarrow I$ with $I$ injective. Up to isomorphism, there are only finitely many subobjects of $I$ which are essential extensions of $A$; we can therefore find a maximal one $E$, say $A \rightarrow E \rightarrow I$. We show that $E$ is injective.
  If not, by theorem \ref{inj:noesse} there is a nontrivial essential extension $e: E \rightarrow G$. Since $I$ is injective, we can extend the embedding $E \rightarrow I$ to $G$, and the resulting extension is monic, i.e. we have expressed $G$ as a subobject of $I$. But then $A \rightarrow E \rightarrow G$ is a strictly larger essential extension of $A$ within $I$, a contradiction. \qed
  \item Because $F$ is injective, we can find a map $h: E \rightarrow F$ with $f = he$. Because $f$ is monic and $e$ essential, $h$ must be monic.
  Similarly we can find a monic $k: F \rightarrow E$. But now, by theorem \ref{endo:monic}, $h$ and $k$ must both be isomorphisms. \qed
\end{enumerate}
\end{proof}

%
%

\end{document}